\documentclass[11pt]{article}
\usepackage[utf8]{inputenc}
\usepackage[T1]{fontenc}

\usepackage{authblk}
\usepackage[colorlinks=true, linkcolor=blue!80!black, citecolor=blue!80!black, urlcolor=blue!80!black]{hyperref}
\usepackage{url}
\usepackage{booktabs}
\usepackage{graphicx}
\usepackage{amssymb,amsmath,amsthm,mathtools,mathrsfs,braket,bbm,bm}
\usepackage{enumitem}
\usepackage{subcaption}
\usepackage{csquotes}
\usepackage{soul}
\usepackage[table]{xcolor}
\usepackage{makecell}
\usepackage{multirow}
\usepackage{pifont}
\usepackage{dblfloatfix}
\usepackage[nameinlink]{cleveref}
\usepackage{moresize}
\makeatletter
\AddToHook{cmd/appendix/before}{\def\cref@section@alias{appendix}\def\cref@subsection@alias{appendix}}
\makeatother

\usepackage[sort,numbers]{natbib}
\usepackage[hyperpageref]{backref}

\usepackage{pgfplots}
\pgfplotsset{compat=1.18}
\usetikzlibrary{arrows.meta, positioning, fit}

\newcommand{\R}{\mathbb{R}}
\newcommand{\N}{\mathbb{N}}
\newcommand{\F}{\mathscr{F}}
\newcommand{\eu}{\mathrm{e}}
\newcommand{\du}{\mathrm{d}}
\DeclareMathOperator*{\argmax}{arg\!\max}
\DeclareMathOperator*{\argmin}{arg\!\min}
\newcommand{\abs}[1]{\left\lvert#1\right\rvert}
\newcommand{\norm}[1]{\lVert#1\rVert}
\newcommand{\schedule}{\bm{\eta}}

\newtheorem{theorem}{Theorem}[section]
\newtheorem{lemma}{Lemma}[section]
\newtheorem*{remark}{Remark}
\newtheorem{definition}{Definition}[section]

\newcommand{\cmark}{\ding{51}}%
\newcommand{\xmark}{\ding{55}}%

\def\arxiv{}

\title{Lower Bounds for Anytime Acceleration of Gradient Descent}
\author[1]{Chung-En Tsai}
\author[1,2]{Ilyas Fatkhullin}
\author[1]{Liang Zhang}
\author[1]{Niao He}
\affil[1]{Department of Computer Science, ETH Zurich}
\affil[2]{ETH AI Center}
\date{\today}

\begin{document}

\maketitle

\begin{abstract}
Recent work suggests that the convergence rate of gradient descent (GD) in smooth convex optimization can be significantly improved by employing large stepsizes that may violate the descent property.
In particular, if the total number of iterations $n$ is given, an $O(n^{-1.271})$ convergence rate can be achieved for both \textit{function value} and squared \textit{gradient norm} minimization.
On the other hand, in the setting of anytime convergence, where $n$ is not known in advance, the best known rates of GD are much slower: $O(n^{-1.119})$ for function value minimization and $O(n^{-1})$ for squared gradient norm minimization.
It remains open whether any of these upper bounds can be improved, as they are far from the classical $\Omega(n^{-2})$ lower bound for any first-order method.
\ifdefined\arxiv
    \par
\fi
In this work, we establish two lower bounds on the anytime convergence of GD.
We show that no positive stepsize schedule can achieve an $o(n^{-1.334})$ anytime rate for \textit{function value} minimization, nor an $o(n^{-1})$ anytime rate for squared \textit{gradient norm} minimization.
The key ingredients of our analysis are novel upper bounds on the number and the magnitude of large stepsizes, derived by analyzing GD on quadratic functions and variants of Huber functions.
Our work provides the first lower bounds for the COLT 2024 open problem posed by Kornowski and Shamir regarding the optimal anytime convergence rates of GD.
\end{abstract}

\newpage

\section{Introduction}

We consider the problem of minimizing a smooth convex function via gradient descent (GD), a canonical setting in convex optimization \cite{nesterov:2018}.
Given a stepsize schedule $\schedule=(\eta_k)_{k\in\N}$ and an initial point $x_1\in\R^d$, GD generates its iterates according to the following update rule:
\[
	x_{k+1} = x_k - \eta_k \nabla f(x_k), \quad\forall k\in\N,
\]
where $f$ is an $L$-smooth convex function.
While most existing works focus on small stepsizes $\eta_k\in(0,2/L)$, a recent line of work has begun exploring the benefits of using large stepsizes that exceed $2/L$ \cite{altschuler:2024b,altschuler:2026,grimmer:2023,wu:2024}.
Motivated by this, we study the optimal convergence rates of GD under any positive, non-adaptive,\footnotemark{} yet not necessarily bounded stepsize schedule.
\footnotetext{That is, the stepsize schedule $\schedule$ is chosen before running GD and cannot depend on past iterates or gradients \cite{altschuler:2026}.}

There are multiple ways of defining the worst-case convergence rate of an iterative algorithm, and we now introduce the two definitions considered in this work.
Let $\mathcal{A}$ be an iterative algorithm, such as GD, and let $(x_n)$ be its iterates.
Let $\norm{\cdot}$ denote the $\ell_2$-norm and let $\F_L(\R^d)$ denote the class of $L$-smooth convex functions on $\R^d$ whose set of minimizers $X^\star_f\coloneqq\argmin_{x\in\R^d}f(x)$ is nonempty.
Without loss of generality, we assume $L=1$.
In the first setting, the convergence rate is defined as
\begin{equation}\label{eq:r_n}
	R_n(\mathcal{A}) \coloneqq \sup_{d\in\N}
    \sup_{f\in\F_1(\R^d)}
    \sup_{x^\star\in X_f^\star}
    \sup_{x_1\in\R^d\setminus X_f^\star}
    \frac{ f(x_{n+1}) - f(x^\star) }{ \frac{1}{2} \norm{x_1 - x^\star}^2 },
\end{equation}
which is the worst-case ratio of the final function value gap to the squared initial distance to the set of minimizers.
In the second setting, the convergence rate is defined as
\begin{equation}\label{eq:g_n}
	G_n(\mathcal{A}) \coloneqq \sup_{d\in\N}
    \sup_{f\in\F_1(\R^d)}
    \sup_{x^\star\in X_f^\star}
    \sup_{x_1\in\R^d\setminus X_f^\star}
    \frac{ \frac{1}{2}\norm{ \nabla f(x_{n+1}) }^2 }{ f(x_1) - f(x^\star) },
\end{equation}
which is the worst-case ratio of the final squared gradient norm to the initial function value gap.
For simplicity, when $\mathcal{A}$ is GD with a stepsize schedule $\schedule=(\eta_k)_{k\in\N}$, we write $R_n(\schedule)$ and $G_n(\schedule)$ instead of $R_n(\mathcal{A})$ and $G_n(\mathcal{A})$, respectively.
The first definition, $R_n$ \eqref{eq:r_n}, is standard in the literature, while the second one $G_n$ \eqref{eq:g_n} has received increasing attention in recent years \cite{kim:2021,nesterov:2012,taylor:2018}.\footnote{It is worth noting that some works consider a variant of $G_n$ in which $f(x_1)-f(x^\star)$ is replaced by $\smash{\frac{1}{2}\norm{x_1-x^\star}^2}$.
Perhaps surprisingly, this modification drastically changes the convergence rate of GD with a constant stepsize from $O(n^{-1})$ to $O(n^{-2})$ \cite{taylor:2018}.
We do not study this alternative definition, and thus our lower bound do not apply to this setting.}

In this work, we focus on the \emph{anytime} convergence rates of GD.
A convergence rate is said to be \emph{anytime} if the stepsize schedule $\schedule$ does not depend on a prescribed stopping time $n$ and the rate applies to all $n\in\N$.
Otherwise, it is said to be \emph{non-anytime}.
We note that the optimal anytime and non-anytime rates of GD differ in general.
For instance, it was recently shown that they differ by a polylogarithmic factor in \emph{non-}smooth convex optimization \cite{kornowski:2025}.

\subsection{Upper Bounds for GD}
For both $R_n$ and $G_n$, and for both anytime and non-anytime convergence, the standard result for GD is the $O(n^{-1})$ rate \cite{levitin:1966,nesterov:2018,taylor:2018}, attained by a small constant stepsize.
We call any rate of order $o(n^{-1})$ an accelerated rate.
Although it was widely believed that GD cannot achieve an accelerated rate, a series of recent breakthroughs showed that such a rate is indeed achievable by using large stepsizes.
An accelerated \emph{non-anytime} convergence rate for $R_n$ was proved concurrently by Altschuler and Parrilo~\cite{altschuler:2024b} and Grimmer et al.~\cite{grimmer:2023}, with the best known rate of order $O(n^{-\log_2\rho})\approx O(n^{-1.271})$, where $\rho=1+\sqrt{2}$ is the silver ratio.
Subsequently, Zhang and Jiang~\cite{zhang:2024} and Grimmer et al.~\cite{grimmer:2025,grimmer:2025a} established the same $O(n^{-\log_2\rho})$ \emph{non-anytime} convergence rate for $G_n$.
To summarize, \emph{non-anytime} acceleration is achievable for both $R_n$ and $G_n$, and their best known rates are identical up to constant factors.

As for the \emph{anytime} convergence rates, Zhang et al.~\cite{zhang:2025} recently established an accelerated rate of order $O\left(n^{-\smash{\frac{2\log_2\rho}{1+\log_2\rho}}}\right)\approx O(n^{-1.119})$ for $R_n$.
This provides a partial answer to the COLT 2024 open problem posed by Kornowski and Shamir~\cite{kornowski:2024} regarding the optimal anytime convergence rate of GD for $R_n$.
Unlike non-anytime convergence, however, no stepsize schedule is known to achieve an accelerated anytime rate for $G_n$.
As we will show later in this work, achieving such a rate is impossible for any positive stepsize schedule, establishing a fundamental separation between anytime and non-anytime convergence rates for $G_n$.

All existing stepsize schedules that achieve an accelerated rate occasionally use stepsizes larger than $2/L$, which may violate the descent property \cite{altschuler:2024b,grimmer:2023,grimmer:2025,grimmer:2025a,zhang:2024,zhang:2025}.
Indeed, using large stepsizes is necessary to achieve an $o(n^{-1})$ convergence rate \cite{das-gupta:2024,kornowski:2024}.
For example, in the silver stepsize schedule of Altschuler and Parrilo~\cite{altschuler:2024b}, one fourth of the stepsizes exceed $2/L$, and the largest one grows as $\Theta(n^{\log_2\rho}/L) \approx \Theta(n^{1.271}/L)$ with the stopping time $n$.
Although large stepsizes may cause temporary increases in the function value or the gradient norm, these schedules nonetheless guarantee an accelerated convergence rate at the final iterate.

\subsection{Lower Bounds for GD}
To the best of our knowledge, for GD with positive stepsize schedules, the only applicable lower bound is the classical $\Omega(n^{-2})$ one \cite{nemirovsky:1983}, which applies to both $R_n$ and $G_n$, and both anytime and non-anytime rates.\footnotemark
\footnotetext{More precisely, for both $R_n$ and $G_n$, to output an $\varepsilon$-approximate minimizer $\hat{x}$, i.e., $f(\hat{x})- f(x^\star) \leq \varepsilon$ or $\norm{ \nabla f(\hat{x}) }^2 \leq \varepsilon$, any first-order method must have an $\Omega(\varepsilon^{-1/2})$ oracle complexity \cite{nemirovsky:1983,nesterov:2018}.
This translates into an $\Omega(n^{-2})$ lower bound on the anytime and non-anytime convergence rates of GD.}
While this classical lower bounds applies broadly to any first-order gradient method, it remains significantly below the best known upper bounds;
see \Cref{fig:sota_results} for a comparison of existing results.

\ifdefined\arxiv
	\newcommand{\myfontsize}{\ssmall}
\else
	\newcommand{\myfontsize}{\scriptsize}
\fi

\begin{figure}[t]
\definecolor{impfill}{RGB}{247,228,216}
\definecolor{achfill}{RGB}{222,241,211}
\definecolor{imptext}{RGB}{176,36,24}
\definecolor{achtext}{RGB}{50,82,32}
\centering
\begin{tikzpicture}[
  scale=0.15,
  line cap=round,
  line join=round
]
\foreach \x in {23,35,45,55,65}{
  \draw[line width=1pt,dash pattern=on 0pt off 4pt] (\x,-15) -- (\x,30);
}

\filldraw[impfill,draw=red!70!black,line width=1pt]
(0,25) rectangle (23,20); %
\filldraw[impfill,draw=red!70!black,line width=1pt]
(0,20) rectangle (35,15);  %
\filldraw[impfill,draw=red!70!black,line width=1pt]
(0,0) rectangle (23,-5); %
\filldraw[impfill,draw=red!70!black,line width=1pt]
(0,-5) rectangle (65,-10);  %

\fill[achfill,draw=green!50!black,line width=1pt] (45,25) rectangle (83,20); %
\fill[achfill,draw=green!50!black,line width=1pt] (55,20) rectangle (83,15);  %
\fill[achfill,draw=green!50!black,line width=1pt] (45,0) rectangle (83,-5); %
\fill[achfill,draw=green!50!black,line width=1pt] (65,-5) rectangle (83,-10);  %

\draw[line width=1.5pt] (0,20) -- (83,20);
\draw[line width=1.5pt] (0,-5) -- (83,-5);
\draw[dashed, line width=1pt,draw=red!80!black,<->] (23,22.5) -- (45,22.5);
\draw[dashed, line width=1pt,draw=red!80!black,<->] (35,17.5) -- (55,17.5);
\draw[dashed, line width=1pt,draw=red!80!black,<->] (23,-2.5) -- (45,-2.5);

\node[anchor=south] at (34,22.5) {\footnotesize\color{imptext} Unknown};
\node[anchor=north] at (45,17.5) {\footnotesize\color{imptext} Unknown};
\node[anchor=south] at (34,-2.5) {\footnotesize\color{imptext} Unknown};

\node[] at (64,22.5) {\myfontsize Altschuler and Parrilo~\cite{altschuler:2024b}};
\node[] at (69,17.5) {\myfontsize Zhang et al.~\cite{zhang:2025}};
\node[] at (64,-2.5) {\myfontsize Zhang and Jiang~\cite{zhang:2024}, Grimmer et al.~\cite{grimmer:2025,grimmer:2025a}};
\node[] at (74,-7.5) {\myfontsize Taylor et al.~\cite{taylor:2018}};

\node[] at (11.5,22.5) {\myfontsize Nemirovsky and Yudin~\cite{nemirovsky:1983}};
\node[] at (17.5,17.5) {\myfontsize This work (\Cref{thm:function_value_impossibility})};
\node[] at (11.5,-2.5) {\myfontsize Nemirovsky and Yudin~\cite{nemirovsky:1983}};
\node[] at (32.5,-7.5) {\myfontsize This work (\Cref{thm:gradient_norm_impossibility})};

\node[anchor=west] at (-1,28) {\textit{$R_n$ \eqref{eq:r_n}}};
\node[anchor=west] at (-1,3) {\textit{$G_n$ \eqref{eq:g_n}}};

\node[anchor=east] at (84,28) {\textit{Non-anytime}};
\node[anchor=east] at (84,12) {\textit{Anytime}};
\node[anchor=east] at (84,3) {\textit{Non-anytime}};
\node[anchor=east] at (84,-13) {\textit{Anytime}};
\node[anchor=north] at (23,-15) {$n^{-2}$};
\node[anchor=north] at (35,-15) {$n^{-1.334}$};
\node[anchor=north] at (45,-15) {$n^{-1.271}$};
\node[anchor=north] at (55,-15) {$n^{-1.119}$};
\node[anchor=north] at (65,-15) {$n^{-1}$};

\end{tikzpicture}
\caption{The best known upper (green) and lower (red) bounds for GD with an arbitrary positive stepsize schedule.
See also \Cref{tab:sota_results,tab:sota_results_first_order_methods} in \Cref{sec:sota_results} for a summary of existing results.}
\label{fig:sota_results}
\end{figure}
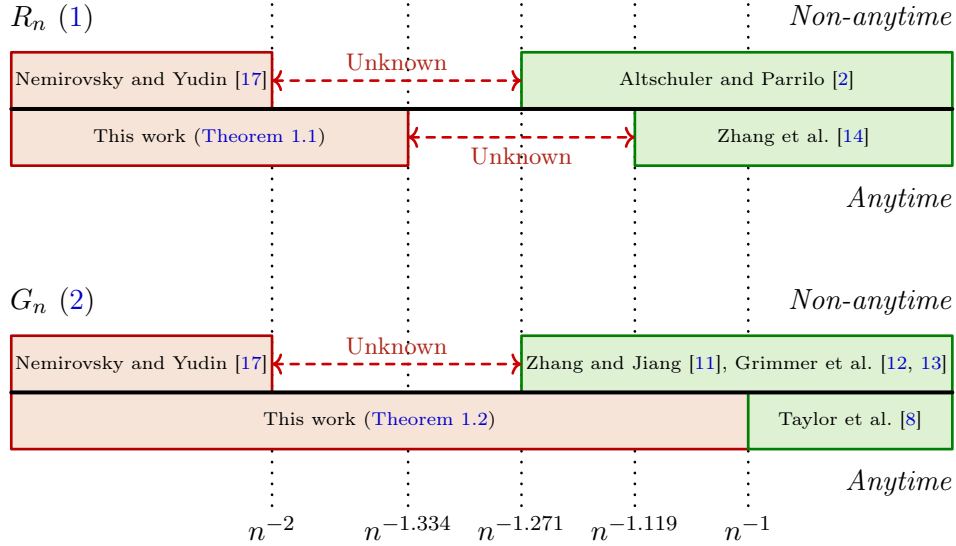

Besides the classical $\Omega(n^{-2})$ lower bound, there are two other lower bounds that apply to \emph{restricted} classes of stepsize schedules of GD.
For both $R_n$ and $G_n$, Grimmer et al.~\cite{grimmer:2025a} established an $\Omega(n^{-\log_2\rho})$ lower bound on the non-anytime convergence rate for all ``basic and composable'' stepsize schedules.
Informally, these are schedules for which a Huber function and the quadratic function $\frac{1}{2}x^2$ are both worst-case instances.
For $R_n$, Kornowski and Shamir~\cite{kornowski:2024} showed that the silver stepsize schedule, which achieves an $O(n^{-\log_2\rho})$ non-anytime convergence rate, does not achieve an $o(n^{-1})$ anytime convergence rate.
In contrast, our results apply to any positive stepsize schedule.

Lastly, from the numerical side, Das Gupta et al.~\cite{das-gupta:2024} minimized $R_n(\schedule)$ over all $\schedule$ for each $n\leq 50$ and found that the resulting convergence rates decay roughly as $O(n^{-1.178})$.
This provides evidence that an $\omega(n^{-2})$ lower bound likely holds for both anytime and non-anytime convergence rates of GD.

\subsection{Contributions and Main Results}
In this work, we establish the following two lower bounds on the anytime convergence rates of GD under arbitrary positive stepsize schedules.
The proofs are provided in \Cref{sec:lower_bounds}.

\begin{theorem}\label{thm:function_value_impossibility}
No stepsize schedule $\schedule\in(0,\infty)^{\N}$ satisfies $R_n(\schedule)=o(n^{-4/3})$.
\end{theorem}

\begin{theorem}\label{thm:gradient_norm_impossibility}
No stepsize schedule $\schedule\in(0,\infty)^{\N}$ satisfies $G_n(\schedule)=o(n^{-1})$.
\end{theorem}

\paragraph{Remarks}
We make two remarks on \Cref{thm:function_value_impossibility,thm:gradient_norm_impossibility}.
First, they do not imply that $R_n(\schedule)=\Omega(n^{-4/3})$ and $G_n(\schedule)=\Omega(n^{-1})$ for all $\schedule\in(0,\infty)^{\N}$, which would be a stronger statement.
More precisely, our theorems only rule out stepsize schedules that achieve these rates for \emph{all} $n\in\N$.
They leave open the possibility that a stepsize schedule $\schedule\in(0,\infty)^{\N}$ achieves $R_n(\schedule)=o(n^{-4/3})$ or $G_n(\schedule)=o(n^{-1})$ for infinitely many, but not all, $n\in\N$.
In contrast, the stronger statement would rule out this possibility entirely.
Second, both \Cref{thm:function_value_impossibility,thm:gradient_norm_impossibility} apply to minimizing univariate functions, as their proofs only construct such functions.

\paragraph{Implications}
We highlight a few implications of \Cref{thm:function_value_impossibility,thm:gradient_norm_impossibility}.
Regarding \Cref{thm:function_value_impossibility}, by combining it with the upper bound of Zhang et al.~\cite{zhang:2025}, the answer to the COLT open problem \cite{kornowski:2024} lies between $n^{-1.334}$ and $n^{-1.119}$.
Furthermore, since accelerated methods achieve an $O(n^{-2})$ anytime rate for $R_n$ \cite{nesterov:1983}, 
GD with a positive stepsize schedule is provably suboptimal among first-order methods.

Regarding \Cref{thm:gradient_norm_impossibility}, we note three consequences.
First, it justifies the absence of an $o(n^{-1})$ anytime convergence rate for $G_n$ in the literature, and proves a special case of the conjecture of Diakonikolas and Wang~\cite{diakonikolas:2022}.\footnote{Diakonikolas and Wang~\cite{diakonikolas:2022} conjectured that every algorithm with the update rule $x_{n+1} = x_1 - \sum_{k=1}^n \beta_{k,n}\nabla f(x_k)$ for all $n\in\N$ satisfies $G_n=\Omega(n^{-1})$.
We show that in the special case where $\beta_{k,n}=\eta_k>0$ is positive and independent of $n$, it is impossible to achieve $G_n=o(n^{-1})$.}
Second, it matches the $O(n^{-1})$ anytime convergence rate achieved by a small constant stepsize \cite{rotaru:2026,taylor:2018}, thereby closing the gap between upper and lower bounds.
Finally, while Zhang and Jiang~\cite{zhang:2024} and Grimmer et al.~\cite{grimmer:2025a} showed that an $O(n^{-1.271})$ rate is achievable in the non-anytime setting, \Cref{thm:gradient_norm_impossibility} shows that this rate is unachievable in the anytime setting.
This quantifies the price of not knowing the stopping time in advance.

\subsection{Technical Contributions}
The main challenge in establishing an $\omega(n^{-2})$ lower bound lies in controlling the \emph{magnitude} and \emph{number} of large stepsizes, which together constrain the progress GD can make.
Except for Corollary~3 of Kornowski and Shamir~\cite{kornowski:2024} which establishes an upper bound on the magnitude in the $R_n$ setting, we are unaware of any existing results that bound the magnitude and number of large stepsizes.
Our technical contributions address exactly these two aspects.

In \Cref{sec:quadratic}, we establish an upper bound on the number of large stepsizes.
Our analysis builds on existing techniques for quadratic optimization, which reduce the problem of determining the optimal worst-case convergence rate to bounding the norm of a polynomial \cite{daspremont:2021,nemirovsky:1992,young:1953}.
Informally speaking, we show that if a stepsize schedule converges on all convex quadratic functions, then the proportion of stepsizes larger than a threshold $t>0$ decays at an $O(1/\sqrt{t})$ rate.
As we are unaware of comparable results in the literature, we believe this bound is of independent interest.

In \Cref{sec:asymmetric_huber}, we extend the upper bound of Kornowski and Shamir~\cite{kornowski:2024} on the magnitude of large stepsizes in two directions.
First, we extend their upper bound to the $G_n$ setting.
Second, while they constructed a Huber-like function that is difficult for GD to optimize with a large \emph{last} stepsize, we generalize their construction to a large stepsize at an arbitrary iteration.

\begin{figure}[t]
\definecolor{MyRed}{RGB}{247,228,216}
\definecolor{MyGreen}{RGB}{222,241,211}
\centering
\begin{tikzpicture}[
	scale=1.0,
    transform shape,
    node distance=1.2cm and 1.7cm,
    box/.style={
        draw,
        rounded corners=0pt,
        minimum width=3.4cm,
        minimum height=1.15cm,
        align=center,
        thick
    },
    smallbox/.style={
        draw,
        rounded corners=0pt,
        minimum width=1.7cm,
        minimum height=0.75cm,
        align=center,
        thick
    },
    arrow/.style={
        -{Latex[length=3mm]},
        thick
    },
    section2/.style={
        draw=red!70!black,
        rounded corners=0pt,
        fill=MyRed,
        fill opacity=0.25,
        thick,
        densely dashed,
        inner sep=0.2cm,
        minimum width=11cm
    },
    section3/.style={
        draw=green!30!black,
        rounded corners=0pt,
        fill=MyGreen,
        fill opacity=0.25,
        thick,
        densely dashed,
        inner sep=0.2cm,
        minimum width=11cm,
    }
]

\node[smallbox] (quadratic) {\small Convex quadratic functions};

\node[box, below left=0.7cm and -2.3cm of quadratic, fill=MyRed, draw=red!70!black] (number)
    {\small \Cref{lem:number_bound}\\
     \small Bound on number of large stepsizes};

\node[box, below=0.5cm of number] (sum)
    {\small \Cref{lem:bound_on_large_step_to_bound_on_sum}\\
     \small Bound on sum of stepsizes\\
     \small in terms of maximum stepsizes};

\node[box, below right=0.7cm and -1.6cm of quadratic] (lb)
    {\small \Cref{lem:convergence_rate_lower_bound_by_sum_of_stepsizes}\\
     \small Connect convergence rates\\
     \small with sum of stepsizes};

\node[box, below=0.8cm of sum] (functionvalue) {\small \Cref{thm:function_value_impossibility}\\
\small Lower Bound for $R_n$};

\node[box, below=2.5cm of lb] (gradientnorm) {\small \Cref{thm:gradient_norm_impossibility}\\
\small Lower Bound for $G_n$};

\draw[arrow] (quadratic.south) -- ++(0,-0.25) -| (number.north);
\draw[arrow] (quadratic.south) -- ++(0,-0.25) -| (lb.north);

\draw[arrow] (number) -- (sum);
\draw[arrow] (sum) -- (functionvalue);
\draw[arrow] (lb.south) -- ++(0,-2.1) -| ([xshift=1cm]functionvalue.north);

\draw[arrow] (lb.south) -- (gradientnorm.north);

\node[
    section2,
    fit=(quadratic)(number)(sum)(lb),
] (sec2) {};

\node[align=left, above left=-0.5cm and 1.5cm of quadratic] (sec2title) {\bf\color{red!70!black}\Cref{sec:quadratic}};

\node[box, below=2.7cm of sec2.south, fill=MyGreen, draw=green!30!black] (magnitude)
    {\small \Cref{lem:asymmetric_huber}\\
    \small Bound on magnitude of large stepsizes};

\node[smallbox, below=0.5cm of magnitude.south] (asym) {A\small symmetric Huber functions};

\draw[arrow] (asym) -- (magnitude);
\draw[arrow] (magnitude.north) -- ++(0,0.5) -| (functionvalue.south);
\draw[arrow] (magnitude.north) -- ++(0,0.5) -| (gradientnorm.south);

\node[
    section3,
    fit=(asym)(magnitude),
] (sec3) {};

\node[align=left, below=7cm of sec2title] {\bf\color{green!30!black}\Cref{sec:asymmetric_huber}};
\end{tikzpicture}

\caption{The key building blocks for the proofs of \Cref{thm:function_value_impossibility,thm:gradient_norm_impossibility}.
The two shaded blocks highlight our main technical contributions (\Cref{lem:number_bound,lem:asymmetric_huber}).}
\label{fig:roadmap}
\end{figure}
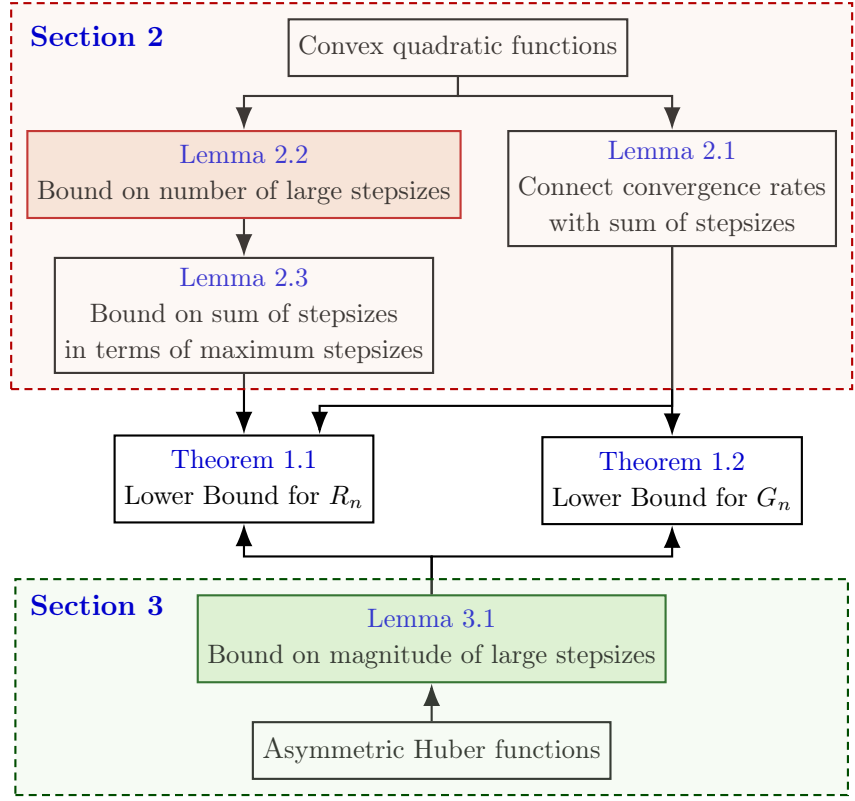

\subsection{Roadmap}

We outline the roadmap of our analysis in \Cref{fig:roadmap}.
In \Cref{sec:quadratic}, we study the convergence rates of GD on convex quadratic functions.
We explain why controlling the number and magnitude of large stepsizes is important, and prove upper bounds on the number of large stepsizes.
In \Cref{sec:asymmetric_huber}, we introduce the asymmetric Huber functions and establish upper bounds on the magnitude of large stepsizes.
In \Cref{sec:lower_bounds}, we combine these results to prove \Cref{thm:function_value_impossibility,thm:gradient_norm_impossibility}.

\paragraph{Notation}
For $n\in\N$, let $[n]\coloneqq\{1,2,\ldots,n\}$.
The set of all infinite sequences of positive numbers is denoted by $(0,\infty)^{\N}$.
For $(a_n), (b_n)\in(0,\infty)^{\N}$, we write $a_n=O(b_n)$ and $b_n=\Omega(a_n)$ if $\limsup_{n\to\infty}(a_n/b_n)<\infty$, or equivalently, $a_n\leq Cb_n$ for all $n\in\N$ and some constant $C>0$.
We write $a_n=o(b_n)$ and $b_n=\omega(a_n)$ if $\lim_{n\to\infty}(a_n/b_n)=0$.
We write $a_n=\Theta(b_n)$ if $a_n=O(b_n)$ and $a_n=\Omega(b_n)$.
A differentiable function $f:\R^d\to\R$ is $L$-smooth if its gradient is $L$-Lipschitz continuous, i.e., $\norm{ \nabla f(x) - \nabla f(y) } \leq L \norm{ x - y }$ for all $x,y\in\R^d$.
For a stepsize schedule $\schedule = (\eta_k)_{k\in\N} \in(0,\infty)^{\N}$, we denote its partial sum by $\eta_{m:n}=\sum_{k=m}^n \eta_k$ for $m\leq n$, and define $\eta_{m:n}=0$ if $m>n$.
We denote the set of real polynomials of degree $n$ by $\mathscr{P}_n$.
The infinity norm of a polynomial $p\in\mathscr{P}_n$ over an interval $[a,b]\subseteq\R$ is defined as $\norm{ p }_{[a,b]}\coloneqq \max_{x\in [a,b]} \abs{p(x)}$.

\section{GD on Convex Quadratic Functions}\label{sec:quadratic}

In this section, we present an in-depth analysis of the convergence rate of GD on convex quadratic functions.
Our goal is to establish an upper bound on the number of large stepsizes in \Cref{lem:number_bound}.

For any $\lambda\in[0,1]$, consider the quadratic function $f(x) = \frac{\lambda}{2}x^2$, which is $1$-smooth, convex, and minimized at $x^\star=0$.
Let $x_1=1$ be the initial point of GD.
Since $f'(x) = \lambda x$, we have $x_{n+1} = \prod_{k=1}^n (1-\eta_k\lambda)$, and thus for all $n\in\N$,
\begin{align*}
	R_n(\schedule) &\geq \lambda x_{n+1}^2 = \lambda \prod_{k=1}^n (1 - \eta_k \lambda )^2,
	\\
	G_n(\schedule) &\geq \lambda x_{n+1}^2 = \lambda \prod_{k=1}^n (1 - \eta_k \lambda )^2.
\end{align*}
By taking the maximum over all $\lambda\in[0,1]$ on the right-hand sides, we conclude that $R_n(\schedule)\geq Q_n(\schedule)$ and $G_n(\schedule) \geq Q_n(\schedule)$ for all $n\in\N$, where
\begin{equation}\label{eq:quadratic_lower_bound}
	Q_n(\schedule)
	\coloneqq \max_{\lambda\in[0,1]}\lambda \prod_{k=1}^n (1 - \eta_k \lambda )^2.
\end{equation}
Deriving lower bounds on $Q_n$ is challenging, as there is no explicit formula for the optimal $\lambda$ that maximizes the polynomial in \eqref{eq:quadratic_lower_bound} when $n$ is large.
In what follows, we derive lower bounds on $Q_n$ by choosing several specific stepsize-dependent values of $\lambda$.
The proofs rely on several results from polynomial approximation theory, which are collected in \Cref{sec:helper_lemmas}.

\subsection{Convergence Rate and Sum of Stepsizes}

We now present the following lemma, which is the starting point of the proofs of \Cref{thm:function_value_impossibility,thm:gradient_norm_impossibility}.

\begin{lemma}\label{lem:convergence_rate_lower_bound_by_sum_of_stepsizes}
For any $\schedule\in(0,\infty)^{\N}$,
\begin{enumerate}[label=(\roman*)]
    \item we have $Q_n(\schedule) \geq \frac{1}{4(1+2\eta_{1:n})}$ for all $n\in\N$.
    \item If $Q_n(\schedule) = o(n^{-1})$, then $\limsup_{n\to\infty}\eta_n = \infty$.
\end{enumerate}
\end{lemma}

\Cref{lem:convergence_rate_lower_bound_by_sum_of_stepsizes} (i) shows that an upper bound on the sum of stepsizes directly translates into a lower bound on the convergence rate.
This motivates our strategy of controlling the magnitude and the number of large stepsizes.
On the other hand, \Cref{lem:convergence_rate_lower_bound_by_sum_of_stepsizes} (ii) explains the necessity of large stepsizes for achieving an accelerated rate.
This result is identical to Theorem~1 of Kornowski and Shamir~\cite{kornowski:2024}, and we include its proof for completeness.

\begin{proof}
For (i), by choosing $\lambda_0 = \frac{1}{1+2\eta_{1:n}}\in[0,1]$ in \eqref{eq:quadratic_lower_bound}, we obtain
\[
	Q_n(\schedule) \geq \lambda_0 \prod_{k=1}^n (1-\eta_k\lambda_0)^2
	\geq \lambda_0 \left( 1 - \lambda_0 \eta_{1:n} \right)^2,
\]
where the last inequality follows from the Weierstrass product inequality.
Since $1-\lambda_0\eta_{1:n} = \frac{1+\eta_{1:n}}{1+2\eta_{1:n}}\geq\frac{1}{2}$, we conclude that $Q_n(\schedule)\geq \frac{\lambda_0}{4} = \frac{1}{4(1+2\eta_{1:n})}$.
For (ii), suppose $Q_n(\schedule)=o(n^{-1})$ but $\limsup_{n\to\infty}\eta_n<\infty$.
By (i), we have $Q_n(\schedule) \geq (4(1+nM))^{-1} = \Omega(n^{-1})$, which contradicts the assumption and completes the proof.
\end{proof}

\begin{remark}
The lower bound in \Cref{lem:convergence_rate_lower_bound_by_sum_of_stepsizes} (i) may remind readers about the two lower bounds, $R_n(\schedule)\geq\frac{1}{1+2\eta_{1:n}}$ and $G_n(\schedule)\geq\frac{1}{1+2\eta_{1:n}}$, both of which can be derived from the Huber function:
\begin{equation}\label{eq:huber}
	f(x) = \begin{cases}
		\frac{1}{2}x^2 &\text{if } \abs{x}\leq\delta, \\
		\delta\abs{x} - \frac{1}{2}\delta^2 &\text{if }\abs{x}\geq\delta,
	\end{cases}
\end{equation}
with $x_1=1$ and some properly chosen $\delta$.
Although both bounds are commonly used in the literature (see, e.g., \cite{drori:2014,grimmer:2025a,taylor:2017a}), we do not use them in our proof.
The reason is that they only improve  \Cref{lem:convergence_rate_lower_bound_by_sum_of_stepsizes} (i) by a constant factor, so incorporating them will not improve the order of our results.
\end{remark}

\subsection{Upper Bound on Number of Large Stepsizes}

Now we present our first technical contribution, an upper bound on the number of large stepsizes.
We first introduce a few definitions.
Let $\schedule\in(0,\infty)^{\N}$ be a stepsize schedule.
For any stopping time $n\in\N$, let $\mu_n \coloneqq  \sum_{k=1}^n \delta_{\eta_k}$ be the counting measure of the first $n$ stepsizes, where $\delta_x$ denotes the Dirac measure at $x$.
Let
\[
	N_n(t) \coloneqq  \mu_n((t,\infty)) = \abs{ \{ k\in[n]: \eta_k > t \} }
\]
denote the number of stepsizes in the first $n$ iterations that are larger than a threshold $t\geq 0$.
This function is non-increasing and piecewise constant.
Note that $N_n(0)=n$ and $N_n(t)=0$ for all $t \geq \max_{k\in[n]}\eta_k$, and we have the trivial upper bound
\begin{equation}\label{eq:trivial_upper_bound_on_number}
	N_n(t) \leq n, \quad\forall t\geq 0.
\end{equation}
We establish a tighter upper bound in the following lemma.

\begin{lemma}\label{lem:number_bound}
For any $\schedule\in(0,\infty)^{\N}$ and any $n\in\N$, we have
\[
    N_n(t) \leq \log Q_n(\schedule) + 2\log(6n) + \frac{8n}{\sqrt{t}-2}, \quad\forall t>8.
\]
\end{lemma}

When $Q_n(\schedule)=O(\text{poly}(n))$, \Cref{lem:number_bound} simplifies to $\smash{N_n(t)=O(\log n + \frac{n}{\sqrt{t}})}$, which significantly improves upon the trivial upper bound \eqref{eq:trivial_upper_bound_on_number}.
We briefly explain the intuition behind \Cref{lem:number_bound}.
Suppose that a stepsize schedule $\schedule$ has many stepsizes exceeding a threshold $t_0\gg 1$.
These large stepsizes contribute a multiplicative factor $(t_0-\lambda)^{2N_n(t_0)}\approx (t_0)^{2N_n(t_0)}$ in \eqref{eq:quadratic_lower_bound}, which grows exponentially in $N_n(t_0)$.
If the remaining small stepsizes cannot compensate for this increase, then $Q_n(\schedule)$ cannot be small.
Consequently, a small $Q_n(\schedule)$ implies an upper bound on $N_n(t_0)$.
With this intuition in mind, we now prove \Cref{lem:number_bound}.

\begin{proof}	
Fix any $\theta > 4$ and decompose $Q_n(\schedule)$ as
\begin{equation}\label{eq:quadratic_lower_bound_small_poly_large_poly}
	Q_n(\schedule)
    = \max_{\lambda\in[0,1]} p_{\leq\theta} (\lambda) \cdot p_{>\theta}(\lambda)
	\geq \max_{\lambda\in[2/\theta,1]} p_{\leq\theta} (\lambda) \cdot p_{>\theta}(\lambda),
\end{equation}
where $p_{\leq\theta}$ and $p_{>\theta}$ are two real polynomials defined as
\begin{align*}
	p_{\leq\theta} (\lambda) &\coloneqq \lambda \prod_{k\in[n]: \eta_k\leq \theta }(1-\lambda \eta_k)^2,\\
	p_{>\theta} (\lambda) &\coloneqq \prod_{k\in[n]: \eta_k > \theta }(1 - \lambda \eta_k)^2.
\end{align*}
For any $\lambda\in[2/\theta,1]$, since all stepsizes in $p_{>\theta}$ satisfy $\eta_k>\theta$, it follows that $\lambda\eta_k>2$ and $\abs{p_{>\theta}(\lambda)} > 1$.
The central challenge is to identify a $\lambda\in[2/\theta,1]$ so that $\abs{p_{\leq\theta}(\lambda)}$ is sufficiently large.
To this end, let
\[
	\lambda^\star \in \argmax_{\lambda\in[2/\theta,1]} \abs{ p_{\leq\theta}(\lambda) }.
\]

\paragraph{\textbf{\textit{Analysis of $p_{\leq\theta}(\lambda^\star)$}}}
Since $p_{\leq\theta}'(0)=1$, we have
\[
	\abs{ p_{\leq\theta}(\lambda^\star) }
	= \max_{\lambda\in[2/\theta,1]} \abs{ p_{\leq\theta}(\lambda) }
	\geq \min_{p\in \mathscr{P}_r: p'(0)=1} \max_{\lambda\in[2/\theta,1]} \abs{ p(\lambda) },
\]
where $r=\deg p_{\leq\theta} = 2(n-N_n(\theta))+1$ and $\mathscr{P}_r$ denotes the set of real polynomials of degree $r$.
Let $\xi = \frac{\sqrt{\theta} + \sqrt{2}}{\sqrt{\theta} - \sqrt{2}} > 1$.
By applying \Cref{lem:minimum_norm_polynomial_derivative}, we have
\[
	\abs{ p_{\leq\theta}(\lambda^\star) } \geq \frac{1-\frac{2}{\theta}}{2r^2} \xi^{-(r-1)}
	\geq \frac{1}{36n^2} \xi^{-(r-1)}
\]
where the last inequality follows from $\theta > 4$ and $r\leq 3n$.
By taking the logarithm on both sides and applying the inequality $\log\xi\leq \xi-1$, we obtain 
\begin{equation}\label{eq:analysis_of_small_polynomial}
\begin{split}
    \log \abs{ p_{\leq\theta}(\lambda^\star) }
    &\geq -(r-1) \log \xi - 2\log (6n) \\
    &= - 2(n-N_n(\theta))\log \xi - 2\log (6n) \\
    &\geq - 2(n-N_n(\theta)) \frac{2\sqrt{2}}{\sqrt{\theta} - \sqrt{2}} - 2\log (6n) \\
    &\geq -\frac{4\sqrt{2}n}{\sqrt{\theta} - \sqrt{2}} - 2\log (6n).
\end{split}
\end{equation}

\paragraph{\textbf{\textit{Analysis of $p_{>\theta}(\lambda^\star)$}}}
Let $\lambda_0=2/\theta$, so $\lambda^\star\geq\lambda_0$.
For every $\eta_k>\theta$, we have $\lambda^\star \eta_k - 1 \geq \lambda_0\eta_k -1\geq 2-1 = 1$.
By invoking the layer-cake representation (\Cref{lem:layer_cake}), we have
\begin{align*}
	\log \abs{ p_{>\theta}(\lambda^\star) }
	&= 2\sum_{k\in[n]:\eta_k>\theta} \log(\lambda^\star\eta_k-1) \\
	&\geq 2\sum_{k\in[n]:\eta_k>\theta} \log(\lambda_0\eta_k-1) \\
	&= 2\int_{(\theta,\infty)} \log ( \lambda_0\eta - 1 )\, \du\mu_n(\eta) \\
	&= 2\int_0^\infty \mu_n\left( \{ \eta\in(\theta,\infty): \log( \lambda_0\eta - 1 ) > t \} \right) \, \du t \\
	&= 2\int_0^\infty \mu_n\left( \left\{ \eta\in(\theta,\infty): \eta > \frac{1 + \eu^t}{\lambda_0} \right\} \right) \, \du t.
\end{align*}
Since $\frac{1 + \eu^t}{\lambda_0} \geq \frac{2}{\lambda_0} = \theta$ for every $t\geq 0$, the last integrand simplifies to
\[
	\mu_n\left( \left\{ \eta\in(\theta,\infty): \eta > \frac{1 + \eu^t}{\lambda_0} \right\} \right)
	= N_n\left( \frac{1 + \eu^t}{\lambda_0} \right).
\]
Performing a change of variable $s=\frac{1 + \eu^t}{\lambda_0}$ gives
\begin{equation}\label{eq:analysis_of_large_polynomial}
\begin{split}
	\log \abs{ p_{>\theta}(\lambda^\star) }
	&\geq 2\int_0^\infty N_n\left( \frac{1 + \eu^t}{\lambda_0} \right) \, \du t \\
	&= 2\int_{\theta}^\infty \frac{N_n(s)}{s-1/\lambda_0}\,\du s \\
	&\geq 2\int_{\theta}^\infty \frac{N_n(s)}{s-1}\,\du s,
\end{split}
\end{equation}
where the last inequality follows from $\lambda_0=2/\theta<1$.

\paragraph{\textbf{\textit{Final step}}}
By taking the logarithm on both sides of \eqref{eq:quadratic_lower_bound_small_poly_large_poly}, we obtain
\[
	\log Q_n(\schedule)
    \geq \log \abs{ p_{\leq\theta}(\lambda^\star) } + \log \abs{ p_{>\theta}(\lambda^\star) }.
\]
Substituting the estimates from \eqref{eq:analysis_of_small_polynomial} and \eqref{eq:analysis_of_large_polynomial} into the above inequality gives
\[
	\log Q_n(\schedule)
	\geq - \frac{4\sqrt{2}n}{\sqrt{\theta}-\sqrt{2}} - 2\log (6n)
	+ 2\int_{\theta}^\infty \frac{N_n(t)}{t-1}\,\du t.
\]
Rearranging the inequality yields
\[
	\int_{\theta}^\infty \frac{N_n(t)}{t-1}\,\du t
	\leq \frac{ \log Q_n(\schedule) }{2} + \log (6n) + \frac{2\sqrt{2}n}{\sqrt{\theta}-\sqrt{2}}.
\]
Finally, for $\theta > 4$, since the integrand $\frac{N_n(t)}{t-1}$ is decreasing, we have
\[
	\int_{\theta}^\infty \frac{N_n(t)}{t-1}\,\du t
	\geq \int_{\theta}^{2\theta} \frac{N_n(t)}{t-1}\,\du t
	\geq (2\theta-\theta)\cdot\frac{N_n(2\theta)}{2\theta-1}
	\geq \frac{1}{2} N_n(2\theta).
\]
Therefore, we have
\[
	\frac{1}{2}N_n(2\theta)
	\leq \frac{ \log Q_n(\schedule) }{2} + \log (6n) + \frac{2\sqrt{2}n}{\sqrt{\theta}-\sqrt{2}}, \quad\forall \theta>4.
\]
Multiplying both sides by $2$ and performing a change of variable $t=2\theta$ give
\[
    N_n(t) \leq \log Q_n(\schedule) + 2\log(6n) + \frac{8n}{\sqrt{t}-2}, \quad\forall t>8.
\]
This completes the proof.
\end{proof}

\subsection{Upper Bound on Sum of Stepsizes}

By \Cref{lem:number_bound}, we establish the following upper bound on the sum of stepsizes in terms of the maximum stepsize.

\begin{lemma}\label{lem:bound_on_large_step_to_bound_on_sum}
For any $\schedule\in(0,\infty)^{\N}$, let $M_n:=\max_{k\in[n]}\eta_k$ be the maximum stepsize in the first $n$ iterations.
If $Q_n(\schedule)=O(1)$, then
\[
	\eta_{1:n} \leq c\cdot \left( n\sqrt{M_n} + M_n \log n + n \right), \quad\forall n\in\N,
\]
for some universal constant $c>0$.
\end{lemma}

Compared with the trivial upper bound $\eta_{1:n}\leq nM_n$, \Cref{lem:bound_on_large_step_to_bound_on_sum} considerably improves the dependence on $M_n$.
Such an improvement is particularly important when $M_n$ is extremely large but such large stepsizes are rare.
For example, in the silver stepsize schedule \cite{altschuler:2024b}, we have $\eta_{1:n} \approx \Theta(n^{1.271})$ and $M_n \approx \Theta(n^{1.271})$.
Lemma~\ref{lem:bound_on_large_step_to_bound_on_sum} implies $\eta_{1:n}= O(n^{1.636})$, whereas the trivial upper bound gives $\eta_{1:n}= O(n^{2.271})$.

\begin{proof}
For any $n\in\N$, by the layer-cake representation (\Cref{lem:layer_cake}),
\begin{align*}
	\eta_{1:n}
	&= \int_{(0,\infty)} \eta \, \du \mu_n(\eta)\\
	&= \int_0^\infty \mu_n\left( \{ \eta\in(0,\infty): \eta > t \} \right)\, \du t \\
	&= \int_0^\infty N_n(t)\,\du t \\
	&= \int_0^{M_n} N_n(t)\,\du t,
\end{align*}
where the last equality follows from $N_n(t)=0$ for all $t\geq M_n$.

Now, if $M_n < 9$, then we have $\eta_{1:n} \leq 9n$, so the lemma holds for any $c>9$.
On the other hand, if $M_n\geq 9$, then
\[
    \eta_{1:n}
    = \int_0^9 N_n(t)\,\du t + \int_9^{M_n} N_n(t)\,\du t
    \leq 9n + \int_9^{M_n} N_n(t)\,\du t.
\]
By \Cref{lem:number_bound}, we have
\[
	N_n(t) = O\left(\log Q_n(\schedule) + \log n + \frac{n}{\sqrt{t}} \right), \quad\forall t\geq 9,\ \forall n\in\N,
\]
where $O$ hides a multiplicative constant independent of $t$ and $n$.
Therefore,
\begin{align*}
    \int_9^{M_n} N_n(t)\,\du t
	&=  \int_9^{M_n} O\left(\log Q_n(\schedule) + \log t + \frac{n}{\sqrt{t}} \right) \,\du t \\
	&= O\left( (M_n-9)\log Q_n(\schedule) + M_n\log n + n\sqrt{M_n} + 1 \right) \\
    &= O\left( M_n + 1 + M_n\log n + n\sqrt{M_n} + 1 \right) \\
	&= O\left( M_n\log n + n\sqrt{M_n} + 1 \right),
\end{align*}
where the third equality follows from $Q_n(\schedule)=O(1)$.
We conclude that
\[
	\eta_{1:n} = O\left( n\sqrt{M_n} + M_n\log n + n \right).
\]
This completes the proof.
\end{proof}

\section{GD on Asymmetric Huber Functions}\label{sec:asymmetric_huber}

In \Cref{sec:quadratic}, we showed that an upper bound on the sum of stepsizes implies a lower bound on the convergence rates (\Cref{lem:convergence_rate_lower_bound_by_sum_of_stepsizes}), and then we upper bounded the sum in terms of the maximum stepsize (\Cref{lem:bound_on_large_step_to_bound_on_sum}).
In this section, we establish upper bounds on large stepsizes.

We introduce the asymmetric Huber functions, which are parameterized by $\varepsilon>0$ and $\delta>0$ and are defined as
\begin{equation}\label{eq:asymmetric_huber}
	f(x) \coloneqq \begin{cases}
		\delta\abs{x} - \frac{ \delta^2 }{2} &\text{if }x\geq\delta, \\
		\frac{1}{2}x^2 &\text{if }-\varepsilon \leq x \leq \delta, \\
		\varepsilon\abs{x} - \frac{ \varepsilon^2 }{2} &\text{if }x\leq-\varepsilon.
	\end{cases}
\end{equation}
When $\varepsilon=\delta$, it recovers the standard Huber function \eqref{eq:huber}.
Since $f$ is quadratic on $[-\varepsilon,\delta]$ and linear outside $[-\varepsilon,\delta]$, it is smooth and convex.

We briefly explain the idea behind our construction;
see \Cref{fig:asymmetric_huber_construction} for an illustration.
Suppose $\schedule\in(0,\infty)^{\N}$ achieves a small $R_n(\schedule)$ and has an extremely large stepsize $\eta_m$ for some $m\in[n]$.
We choose $(\varepsilon,\delta)$ so that GD is in the quadratic region at the $m$-th update and in the linear regions at all other updates.
Since the quadratic region has large curvature, GD ``overshoots'' at the $m$-th update.
In contrast, since the linear regions have zero curvature, GD can only make limited progress at all other updates.
Consequently, if $\eta_m$ is so large that the remaining stepsizes cannot compensate for the overshooting, then GD converges slowly.
In other words, if $R_n(\schedule)$ is small, then $\eta_m$ cannot be arbitrarily large.

\begin{figure}[t]
\centering
\begin{tikzpicture}
\begin{axis}[
  axis lines=center,
  xlabel={$x$},
  ylabel={$y$},
  xlabel style={right},
  ylabel style={above},
  xmin=-3.4, xmax=4.2,
  ymin=-0.5, ymax=5.0,
  xtick={-1, 2, 3.4},
  xticklabels={$-\varepsilon$, $\delta$, $1$},
  ytick=\empty,
  width=9cm,
  height=6cm,
  axis line style={-{Stealth[scale=1.5]}},
  tick style={draw=none},
  clip=false,
]

\addplot[black, thick, domain=-3.3:-1, samples=2] {-x - 0.5};
\addplot[black, thick, domain=-1:2, samples=80] {x*x/2};
\addplot[black, thick, domain=2:3.9, samples=2] {2*x - 2};

\addplot[black, only marks, mark=*, mark size=1.5pt]
  coordinates {(-1, 0.5) (2, 2.0)};

\addplot[dashed, thin] coordinates {(-1, 0) (-1, 0.5)};
\addplot[dashed, thin] coordinates {( 2, 0) ( 2, 2.0)};

\addplot[dashed, thin] coordinates {( 3.4, 0) ( 3.4, 4.8)};

\addplot[red!80!black, only marks, mark=*, mark size=1.5pt]
  coordinates {(3.4, 4.8) (2.7, 3.4) (2, 2) (-2.9, 2.4) (-2.0, 1.5)};

\node[red!80!black,right] at (axis cs: 3.5, 4.8) {$x_1$};
\node[red!80!black,right] at (axis cs: 2.8, 3.4) {$x_2$};
\node[red!80!black,right] at (axis cs: 2.1, 2) {$x_3$};
\node[red!80!black,below] at (axis cs:-3.0, 2.3) {$x_4$};
\node[red!80!black,below] at (axis cs:-2.1, 1.4) {$x_5$};

\draw[-{Stealth[length=5pt]}, red!80!black, thick]
  (axis cs:3.4, 4.8) to[bend right=25] (axis cs:2.7, 3.4);
\draw[-{Stealth[length=5pt]}, red!80!black, thick]
  (axis cs:2.7, 3.4) to[bend right=25] (axis cs:2, 2);
\draw[-{Stealth[length=5pt]}, red!80!black, thick]
  (axis cs:2, 2) to[bend right=30] (axis cs:-2.9, 2.4);
\draw[-{Stealth[length=5pt]}, red!80!black, thick]
  (axis cs:-2.9, 2.4) to[bend left=25] (axis cs:-2.0, 1.5);

\node[] at (-3.2,3.2) {$f(x)$};
\node[red!80!black] at (-0.8, 3.4) {$\eta_m$};

\end{axis}
\end{tikzpicture}

\caption{Visualization of our construction with $m=3$ and $n=4$.}
\label{fig:asymmetric_huber_construction}

\end{figure}
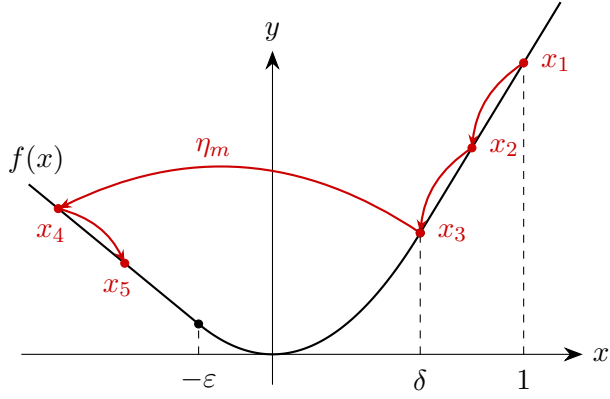

\subsection{Upper Bounds on Magnitude of Large Stepsizes}

The next lemma formalizes the above intuition.

\begin{lemma}\label{lem:asymmetric_huber}
For any $\schedule\in(0,\infty)^{\N}$, any $n\in\N$, and any $m\in[n]$ such that $\eta_m > 1$, we have
\begin{equation}\label{eq:asymmetric_huber_part_i}
\begin{split}
	R_n(\schedule) &\geq \frac{ (\eta_m-1)^2 }{ ( 1 + \eta_{1:m-1} )^2 ( 1 + 2 \eta_{m+1:n} ) }, \\
	G_n(\schedule) &\geq \frac{ (\eta_m-1)^2 }{ ( 1 + 2 \eta_{1:m-1} ) ( 1 + \eta_{m+1:n} )^2 }.
\end{split}
\end{equation}
In particular, we have for any $n\in\N$,
\begin{equation}\label{eq:asymmetric_huber_part_ii}
\begin{split}
	\eta_n &\leq 1 + (1 + \eta_{1:n-1})\sqrt{R_n(\schedule)}, \\
	\eta_n &\leq 1 + \sqrt{1 + 2\eta_{1:n-1}}\sqrt{G_n(\schedule)}.
\end{split}
\end{equation}
\end{lemma}

The assumption $\eta_m>1$ ensures that GD overshoots in the quadratic region.
As expected, the lower bounds in \Cref{lem:asymmetric_huber} grow with $(\eta_m-1)^2$ and decay with all other stepsizes.

\begin{remark}
\Cref{lem:asymmetric_huber} possesses two interesting properties.
The lower bounds for $R_n$ and $G_n$ differ, and they are \emph{not} permutation-invariant:
Permuting the first $n$ stepsizes may change the values of the lower bounds.
This stands in contrast to the lower bounds derived from the convex quadratic functions \eqref{eq:quadratic_lower_bound}.
There, the two lower bounds are identical and permutation-invariant.
\end{remark}

\begin{proof}
We consider the asymmetric Huber function \eqref{eq:asymmetric_huber} with $\delta= \frac{ 1 }{ 1+\eta_{1:m-1} } > 0$ and some
\begin{equation}\label{eq:eps_range}
	0< \varepsilon \leq \frac{  (\eta_m-1)\delta }{ 1 + \eta_{m+1:n} },
\end{equation}
which will be determined later.
Since $\eta_m>1$ and $\delta>0$, the range of $\varepsilon$ is nonempty.
It can be checked that $f$ is $1$-smooth convex and has a unique minimizer at $x^\star=0$.

Now we let $x_1 = 1 > \delta$ be the initial point of GD and compute its trajectory.
By induction, it can be shown that for all $k\in[m]$, we have $x_k\geq\delta$ and $x_k = 1 - \eta_{1:k-1}\delta$.
In particular, $x_m = 1 - \eta_{1:m-1}\delta = \delta$ by the definition of $\delta$.
In the $m$-th step, $x_{m+1} = x_m - \eta_m f'(x_m) = \delta - \eta_m\delta = -(\eta_m-1)\delta$.
Note that
\begin{align*}
	x_{m+1} = -(\eta_m-1)\delta \leq -\frac{  (\eta_m-1)\delta }{ 1 + \eta_{m+1:n} } \leq -\varepsilon,
\end{align*}
where the last step follows from \eqref{eq:eps_range}.
Lastly, by induction, we have $x_k\leq-\varepsilon$ and $x_{k} = x_{m+1} + \eta_{m+1:k-1}\varepsilon$ for every $k=m+1,\ldots,n+1$.

We are now ready to derive lower bounds on $R_n$ and $G_n$.
For $R_n$, by completing the square with respect to $\varepsilon$, we have
\begin{align*}
	f(x_{n+1})
	&= \varepsilon\abs{ x_{n+1} } - \frac{ \varepsilon^2 }{2}  \\
	&= -\varepsilon( x_{m+1} + \eta_{m+1:n}\varepsilon ) - \frac{ \varepsilon^2 }{2} \\
	&= - \frac{1}{2} \left( 1 + 2\eta_{m+1:n} \right) \varepsilon^2 - \varepsilon x_{m+1} \\
	&= - \frac{1}{2} \left( 1 + 2\eta_{m+1:n} \right) \left( \varepsilon + \frac{x_{m+1}}{1+2\eta_{m+1:n}} \right)^2 + \frac{ x_{m+1}^2 }{ 2 ( 1+2\eta_{m+1:n} ) }.
\end{align*}
By choosing $\varepsilon = -\frac{ x_{m+1} }{ 1 + 2\eta_{m+1:n} } = \frac{  (\eta_m-1)\delta }{ 1 + 2\eta_{m+1:n} } \leq \frac{  (\eta_m-1)\delta }{ 1 + \eta_{m+1:n} }$ that maximizes $f(x_{n+1})$, we obtain
\[
	f(x_{n+1}) = \frac{ x_{m+1}^2 }{ 2 ( 1+2\eta_{m+1:n} ) }
	= \frac{ (\eta_m-1)^2 }{ 2 (1 + \eta_{1:m-1})^2( 1+2\eta_{m+1:n} ) }.
\]
Since $\abs{x_1-x^\star}=1$, we conclude that
\[
	R_n(\schedule) \geq \frac{ (\eta_m-1)^2 }{ (1 + \eta_{1:m-1})^2( 1+2\eta_{m+1:n} ) }.
\]
For $G_n$, we have
\[
	G_n(\schedule)
	\geq \frac{ \frac{1}{2}\norm{ \nabla f(x_{n+1}) }^2 }{ f(x_1) - f(x^\star) }
	= \frac{ \varepsilon^2 }{ 2\delta - \delta^2 }.
\]
By choosing $\varepsilon = \frac{  (\eta_m-1)\delta }{ 1 + \eta_{m+1:n} }$ to maximize the lower bound, we obtain
\[
	G_n(\schedule)
	\geq \frac{ (\eta_m-1)^2\delta^2 }{ (2\delta-\delta^2)(1+\eta_{m+1:n})^2 }
	= \frac{ (\eta_m-1)^2 }{ (1 + 2\eta_{1:m-1})(1+\eta_{m+1:n})^2 },
\]
where the last equality follows from the definition of $\delta=\frac{1}{1+\eta_{1:m-1}}$.
This proves \eqref{eq:asymmetric_huber_part_i}.

It remains to prove \eqref{eq:asymmetric_huber_part_ii}.
If $\eta_n\geq 1$, then \eqref{eq:asymmetric_huber_part_ii} follows by taking $m=n$ in \eqref{eq:asymmetric_huber_part_i} and rearranging the inequalities.
If $\eta_n\leq 1$, then \eqref{eq:asymmetric_huber_part_ii} holds trivially.
This completes the proof.
\end{proof}

\section{Proofs of Main Results}\label{sec:lower_bounds}

This section presents the proofs of \Cref{thm:function_value_impossibility,thm:gradient_norm_impossibility}.
We restate the main theorems here for convenience.

\begin{theorem}\label{thm:function_value_impossibility_restate}
No stepsize schedule $\schedule\in(0,\infty)^{\N}$ satisfies $R_n(\schedule)=o(n^{-4/3})$.
\end{theorem}

\begin{theorem}\label{thm:gradient_norm_impossibility_restate}
No stepsize schedule $\schedule\in(0,\infty)^{\N}$ satisfies $G_n(\schedule)=o(n^{-1})$.
\end{theorem}

The stronger lower bound for $G_n$ compared to $R_n$ stems from the differing bounds in \Cref{lem:asymmetric_huber}.
More specifically, by setting $m=n$ in \Cref{lem:asymmetric_huber}, we obtain
\[
	R_n(\schedule) \geq \frac{ (\eta_n-1)^2 }{ ( 1 + \eta_{1:n-1} )^2}
    = \frac{ (\eta_n-1)^2 }{ 1 + 2\eta_{1:n-1} + \eta_{1:n-1}^2},
    \quad
	G_n(\schedule) \geq \frac{ (\eta_n-1)^2 }{ 1 + 2 \eta_{1:n-1} }.
\]
For the same stepsize schedule $\schedule$, since $\eta_{1:n-1}>0$, the lower bound on $G_n$ is always larger than that on $R_n$.

\subsection{Proof of Theorem~\ref{thm:function_value_impossibility_restate}}

Suppose for contradiction that there exists $\schedule\in(0,\infty)^{\N}$ such that GD achieves an $o(n^{-4/3})$ anytime convergence rate for $R_n$, i.e., $R_n(\schedule) = o( n^{-4/3} )$.
Consider the set of all indices at which the maximum stepsize occurs:
\[
	\mathcal{N} = \left\{ n\in\N: \eta_n = \max_{k\in[n]}\eta_k \right\}.
\]
Since $Q_n(\schedule)\leq R_n(\schedule) = o(n^{-4/3})$, by \Cref{lem:convergence_rate_lower_bound_by_sum_of_stepsizes} (ii), we have $\limsup_{n\to\infty}\eta_n = \infty$, implying $\abs{\mathcal{N}}=\infty$.
In the following, we write $a_n=O_{\mathcal{N}}(b_n)$ and $b_n=\Omega_{\mathcal{N}}(a_n)$ if there exists a constant $C>0$ such that $a_n\leq C\cdot b_n$ for all $n\in\mathcal{N}$.

By applying \Cref{lem:bound_on_large_step_to_bound_on_sum} to $n\in\mathcal{N}$, we have
\[
	\eta_{1:n} = O_{\mathcal{N}}\left( n\sqrt{\eta_n} + \eta_n \log n + n \right),
\]
and by \Cref{lem:asymmetric_huber}, we have $\eta_n = O\left( 1 + \eta_{1:n}\sqrt{R_n(\schedule)} \right)$.
Combining them yields
\[
	\eta_{1:n} = O_{\mathcal{N}}\left( n\sqrt{ \eta_{1:n} }R_n(\schedule)^{1/4} + \eta_{1:n}\sqrt{R_n(\schedule)}\log n + n \right).
\]
Dividing both sides by $\eta_{1:n}$, we obtain
\[
	\frac{ nR_n(\schedule)^{1/4} }{ \sqrt{\eta_{1:n}} }  + \sqrt{R_n(\schedule)}\log n + \frac{n}{\eta_{1:n}} = \Omega_{\mathcal{N}}(1).
\]
By \Cref{lem:convergence_rate_lower_bound_by_sum_of_stepsizes} (i) and $R_n(\schedule)\geq Q_n(\schedule)$, we have $\eta_{1:n} = \Omega(R_n(\schedule)^{-1})$, which implies
\[
	n R_n(\schedule)^{3/4} + \sqrt{R_n(\schedule)}\log n + nR_n(\schedule) = \Omega_{\mathcal{N}}(1).
\]
However, by substituting the assumption $R_n(\schedule)=o(n^{-4/3})$ into the left hand side of the above equation, we obtain
\[
	n R_n(\schedule)^{3/4} + \sqrt{R_n(\schedule)}\log n + nR_n(\schedule) = o\left(1 + n^{-2/3}\log n + n^{-1/3}\right)
	= o(1).
\]
Since the intersection of $o(1)$ and $\Omega_{\mathcal{N}}(1)$ is empty, this leads to a contradiction and proves the theorem.

\subsection{Proof of Theorem~\ref{thm:gradient_norm_impossibility_restate}}

For the sake of contradiction, suppose that there exists $\schedule\in(0,\infty)^{\N}$ such that $G_n(\schedule) = o(n^{-1})$.
It follows from \Cref{lem:convergence_rate_lower_bound_by_sum_of_stepsizes} (i) that $\eta_{1:n}=\omega(n)$.
Applying \Cref{lem:asymmetric_huber}, we obtain
\[
	\eta_k \leq 1 + \sqrt{1+2\eta_{1:k-1}}\sqrt{G_k(\schedule)},\quad\forall k\in\N.
\]
Summing from $k=1$ to $n$ yields
\[
	\eta_{1:n} \leq n + \sum_{k=1}^n \sqrt{1+2\eta_{1:k-1}}\sqrt{G_k(\schedule)}
	\leq n + \sqrt{1+2\eta_{1:n}} \sum_{k=1}^n \sqrt{G_k(\schedule)}.
\]
Rearranging the above inequality gives
\[
	\sum_{k=1}^n \sqrt{G_k(\schedule)} \geq \frac{ \eta_{1:n} - n}{ \sqrt{1+2\eta_{1:n}} }
	= \sqrt{\eta_{1:n}}\cdot\frac{1- \frac{n}{\eta_{1:n}}}{\sqrt{\frac{1}{\eta_{1:n}} + 2}}
	= \omega(\sqrt{n}),
\]
where the final equality follows from $\eta_{1:n}=\omega(n)$.
However, since $\sqrt{G_n(\schedule)} = o(n^{-1/2})$, we have $\sum_{k=1}^n \sqrt{G_k(\schedule)} = O(\sqrt{n})$.
This leads to the relationship $O(\sqrt{n}) \geq \omega(\sqrt{n})$, which is a contradiction and completes the proof.

\section{Discussion}

\subsection{Necessary Conditions for Fast Stepsize Schedules}

Our lower bounds are not merely negative results.
They also provide necessary conditions for fast-converging stepsize schedules, which we hope will inspire the design of better ones.
For instance, suppose that there exists a hypothetical stepsize schedule $\schedule\in(0,\infty)^{\N}$ that attains an $O(n^{-4/3})$ anytime convergence rate for $R_n$.
What must $\schedule$ look like?

Let $\mathcal{N} = \smash{\left\{ n\in\N: \eta_n = \max_{k\in[n]}\eta_k \right\}}$ be the set of indices at which maximum stepsize is attained.
For every $n\in\mathcal{N}$, we have $\eta_{1:n}=\Omega(n^{4/3})$ by \Cref{lem:convergence_rate_lower_bound_by_sum_of_stepsizes}.
Combining this with \Cref{lem:bound_on_large_step_to_bound_on_sum}, which asserts $\eta_{1:n} = O(n\sqrt{\eta_n} + \eta_n\log n)$, we obtain $\eta_n=\Omega(n^{2/3})$.
On the other hand, by \Cref{lem:bound_on_large_step_to_bound_on_sum} and \Cref{lem:asymmetric_huber}, we have
\begin{align*}
    \eta_n
    &= O\left(\eta_{1:n}\sqrt{R_n(\schedule)} \right) \\
    &= O\left(n\sqrt{\eta_n R_n(\schedule)} + \eta_n\log n \sqrt{R_n(\schedule)}\right) \\
    &= O\left(n^{1/3}\sqrt{\eta_n} + \frac{\eta_n\log n}{n^{2/3}}\right),
\end{align*}
which implies $\eta_n = O(n^{2/3})$.
It then follows from \Cref{lem:bound_on_large_step_to_bound_on_sum} that $\eta_{1:n} = O(n^{4/3})$ for $n\in\mathcal{N}$.
We therefore conclude that $\eta_n = \Theta(n^{2/3})$ and $\eta_{1:n} = \Theta(n^{4/3})$ for all $n\in\mathcal{N}$.

\subsection{Tightness of Lemma~\ref{lem:number_bound}}\label{sec:tightness_of_number_bound}

The upper bound on the number of large stepsizes (\Cref{lem:number_bound}) is one of the key lemmas in our proofs.
One may wonder whether this lemma can be improved without additional assumptions, as such an improvement would in turn improve \Cref{thm:function_value_impossibility}.
Unfortunately, the answer is likely no.
For each $n\in\N$, there exist stepsize schedules whose number of large stepsizes $N_n(t)$ matches the upper bound in \Cref{lem:number_bound} up to logarithmic terms, implying that the lemma is almost tight.
In the following, we describe the construction of such stepsize schedules.

First, note that $Q_n(\schedule) = \norm{ q_{\schedule} }_{[0,1]}$ only depends on the first $n$ stepsizes and can be written as the infinity norm of the polynomial $q_{\schedule}(\lambda) = \lambda\prod_{k=1}^n(1-\eta_k\lambda)^2$.
It is known that for any stepsize schedule $\schedule\in(0,\infty)^{\N}$ satisfying
\begin{equation}\label{eq:chebyshev_stepsizes}
    \{ \eta_k: k\in[n] \} = \{ \gamma_{n,k}: k\in[n] \},
    \quad
    \gamma_{n,k} \coloneqq \left(\sin\left(\frac{n-k+1}{2n+1}\pi\right) \right)^{\smash{-2}},
\end{equation}
we have $Q_n(\schedule) = (2n+1)^{-2}$ \cite{nemirovsky:1992}.
The set $\{ \gamma_{n,k}: k\in[n] \}$ contains the reciprocals of the roots of the polynomial $x^{-1/2}T_{2n+1}(\sqrt{x})$, where $T_n$ is the $n$-th Chebyshev polynomial of the first kind (see \Cref{def:chebyshev_polynomial}).
Since permuting the first $n$ stepsizes does not affect the value of $Q_n(\schedule)$, the condition \eqref{eq:chebyshev_stepsizes} only requires that the two sets are equal, regardless of the order of the stepsizes.

For $\schedule$ satisfying \eqref{eq:chebyshev_stepsizes}, we have
\begin{align*}
    N_n(t)
    &= \left\lvert \left\{ k\in[n]:  \sin\left(\frac{n-k+1}{2n+1}\pi\right) < \frac{1}{\sqrt{t}} \right\} \right\rvert \\
    &= \left\lvert \left\{ k\in[n]: k > n + 1 - \frac{2n+1}{\pi}\arcsin\left( \frac{1}{\sqrt{t}} \right) \right\} \right\rvert \\
    &\geq n - \left( n + 1 - \frac{2n+1}{\pi}\arcsin\left( \frac{1}{\sqrt{t}} \right) \right)\\
    &= \frac{2n+1}{\pi}\arcsin\left( \frac{1}{\sqrt{t}} \right) - 1.
\end{align*}
Since $\arcsin x \geq x$ for all $x\in[0,1]$, we conclude that
\[
    N_n(t) \geq \frac{2n+1}{\pi\sqrt{t}} - 1 = \Omega\left( \frac{n}{\sqrt{t}} \right), \quad\forall t> 0.
\]
This lower bound matches the upper bound $N_n(t)=O(\log n + \smash{\frac{n}{\sqrt{t}}})$ in \Cref{lem:number_bound} up to an $O(\log n)$ additive term and constant multiplicative factors.
In particular, when $t=O\left(\smash{\frac{n^2}{\log^2 n}}\right)$, the $O(\log n)$ term is dominated by the $O(\frac{n}{\sqrt{t}})$ term, so \Cref{lem:number_bound} is tight up to constant factors.

\subsection{Non-Anytime Convergence Rates}\label{sec:non_anytime_rate}

In view of \Cref{thm:function_value_impossibility,thm:gradient_norm_impossibility}, it is natural to ask whether our techniques also suffice to establish an $\omega(n^{-2})$ lower bound on the non-anytime convergence rates.
In this subsection, we present numerical results suggesting that the answer is likely negative.
Recall that for the non-anytime convergence rates, the stepsize schedule is allowed to depend on the prescribed stopping time $n$.

The proofs of \Cref{thm:function_value_impossibility,thm:gradient_norm_impossibility} rely on two families of functions:
the quadratic functions in \Cref{sec:quadratic} and the asymmetric Huber functions in \Cref{sec:asymmetric_huber}.
We have seen in \Cref{sec:tightness_of_number_bound} that if $\schedule$ satisfies the condition \eqref{eq:chebyshev_stepsizes}, then $Q_n(\schedule) = O(n^{-2})$.
Therefore, the quadratic functions alone are not enough to establish an $\omega(n^{-2})$ lower bound on the non-anytime convergence rates.

Combining $Q_n(\schedule)$ with \Cref{lem:asymmetric_huber} is still insufficient.
The two lower bounds in \Cref{lem:asymmetric_huber} are
\[
	\frac{ (\eta_m-1)^2 }{ ( 1 + \eta_{1:m-1} )^2 ( 1 + 2 \eta_{m+1:n} ) },\quad \frac{ (\eta_m-1)^2 }{ ( 1 + 2 \eta_{1:m-1} ) ( 1 + \eta_{m+1:n} )^2 }.
\]
These are small when, for every $m$ with $\eta_m>1$, the partial sums $\eta_{1:m-1}$ and $\eta_{m+1:n}$ are large.
Based on this observation and the fact that $\gamma_{n,k}$ is increasing in $k$, we construct a permutation of $\{ \gamma_{n,k} : k\in[n] \}$ by placing the largest stepsize $\gamma_{n,n}$ in the middle and then distributing the remaining values alternately to the left and right of $\gamma_{n,n}$.
This leads to the following stepsize schedule:
For every $n\in\N$, let $\schedule^{(n)}\in(0,\infty)^{\N}$ be a stepsize schedule whose first $n$ terms are
\begin{equation}\label{eq:permutation}
	\begin{cases}
		\left(\gamma_{n,1}, \gamma_{n,3}, \cdots, \gamma_{n,n-1}, \gamma_{n,n}, \gamma_{n,n-2}, \ldots, \gamma_{n,4}, \gamma_{n,2} \right) &\text{for odd }n, \\
		\left(\gamma_{n,2}, \gamma_{n,4}, \cdots, \gamma_{n,n-1}, \gamma_{n,n}, \gamma_{n,n-2}, \ldots, \gamma_{n,3}, \gamma_{n,1} \right) &\text{for even }n.
	\end{cases}
\end{equation}
and whose remaining terms are $0$.
Since $\schedule^{(n)}$ satisfies the condition \eqref{eq:chebyshev_stepsizes}, we have  $Q_n(\schedule^{(n)}) = O(n^{-2})$.
We numerically evaluate $Q_n$ and the lower bounds in \Cref{lem:asymmetric_huber} on $\schedule^{(n)}$ and present the results in \Cref{fig:chebyshev_lower_bounds}.
Notably, all lower bounds decay roughly at an $O(n^{-2})$ rate.
This suggests that our analysis cannot establish an $\omega(n^{-2})$ lower bound on the non-anytime convergence rates.
Moreover, deriving such a lower bound remains challenging, as any successful approach must explain why \emph{all} permutations of $\{ \gamma_{n,k} : k\in[n] \}$ fail to achieve fast convergence.

\begin{figure}[t]
        \centering
        \includegraphics[scale=0.55]{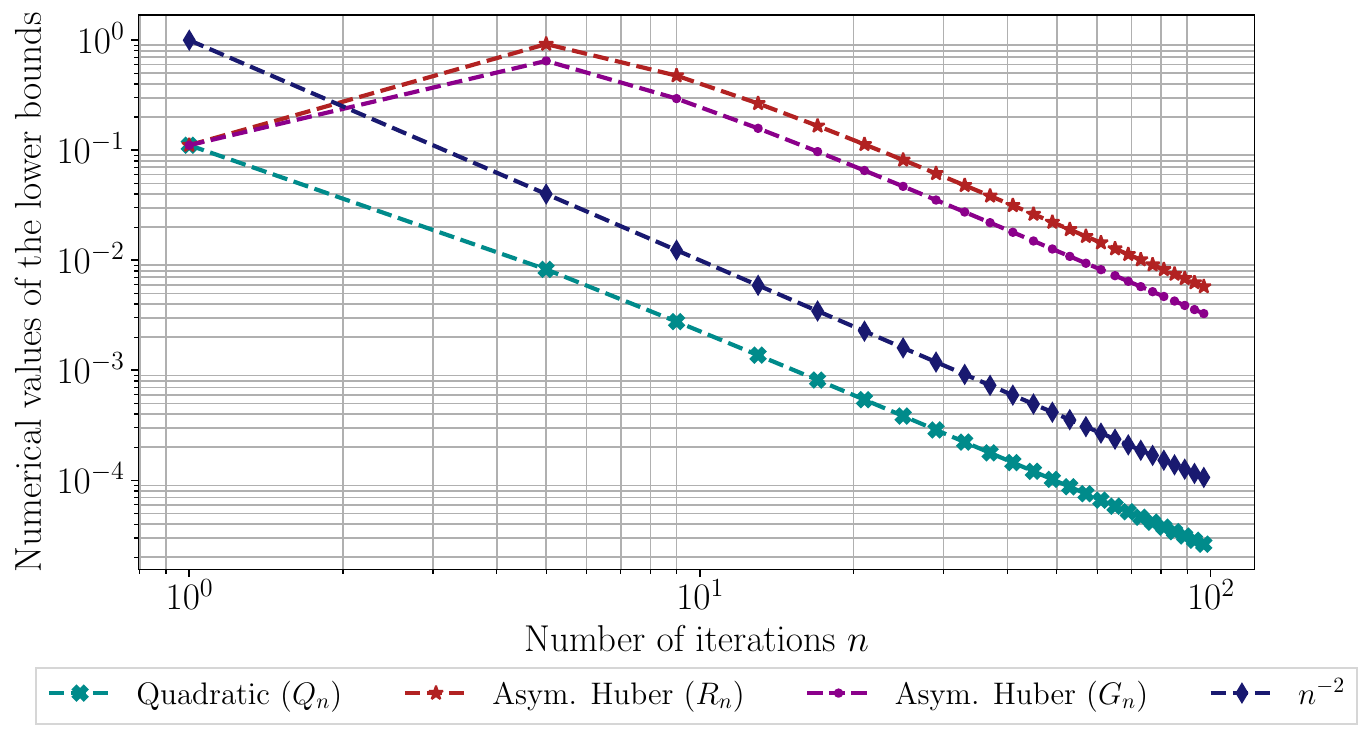}
        \caption{Numerical values of the lower bounds evaluated on the stepsize schedules \eqref{eq:permutation}.
        For the quadratic lower bound, we evaluate the polynomial in $Q_n(\schedule^{(n)})$ at 50 uniform points on $[0,1]$ and report the largest value.
        For the asymmetric Huber lower bound (\Cref{lem:asymmetric_huber}), the largest lower bound over all valid choices of $m\in[n]$ is reported.}
        \label{fig:chebyshev_lower_bounds}
\end{figure}

\subsection{Beyond Last Iterate and Positive Stepsizes}

Our results apply specifically to the last iterate of GD with an arbitrary non-adaptive, positive stepsize schedule.
Proving lower bounds for averaged iterates, negative stepsizes, or adaptive stepsize rules such as AdaGrad \cite{duchi:2011}, is beyond the scope of this work.
While such extensions may be feasible, their technical difficulty should not be underestimated.
For example, Luner and Grimmer~\cite{luner:2025} studied the benefits of averaging and extrapolation in GD, while Shugart and Altschuler~\cite{shugart:2025} demonstrated that incorporating negative stepsizes can improve the convergence of gradient descent-ascent in convex-concave problems.
We leave these extensions to future work.

\section{Conclusions}

In this work, we complement recent advances in the acceleration of GD by proving two lower bounds.
Our results indicate that no positive stepsize schedule can achieve $o(n^{-4/3})$ anytime acceleration for $R_n$, nor $o(n^{-1})$ anytime acceleration for $G_n$.
The key technical contributions underlying our proofs are novel upper bounds on the number and the magnitude of large stepsizes, derived from analyzing GD on quadratic and asymmetric Huber functions, respectively.
These results provide necessary conditions for fast-converging stepsize schedules, and we hope they will inspire the design of new ones.

Several gaps between the best known lower and upper bounds remain.
In particular, finding the optimal \textit{anytime} convergence rate of GD under any positive stepsize schedule in terms of the function value gap, $R_n$, remains open.
Our results together with the upper bound of Zhang et al.~\cite{zhang:2025} indicate that the answer lies between $n^{-1.334}$ and $n^{-1.119}.$
Another important question is deriving an $\omega(n^{-2})$ lower bound for the \textit{non-anytime} convergence rate of GD.
As noted in \Cref{sec:non_anytime_rate}, our technique appears insufficient to establish such a result, and doing so will likely require new constructions and analysis. We leave the investigation of these questions for future work.

\bibliographystyle{unsrtnat}
\bibliography{./ref}

\newpage
\appendix

\section{Best Known Convergence Rates of GD and First-Order Methods}\label{sec:sota_results}

\aboverulesep=0ex
\belowrulesep=0ex

\begin{table}[h]
\centering
\caption{Summary of the best known convergence rates of GD with an arbitrary positive stepsize schedule in smooth convex optimization.
Only for the anytime convergence rate of $G_n$ do the upper and lower bounds coincide.
}
\renewcommand\arraystretch{2}
\begin{tabular}{@{} c|c|c|c }
\toprule
Setting & Anytime & Upper bound & Lower bound \\
\midrule
\multirow{2}{*}{$R_n$} & \xmark & \makecell{\\[-8pt]$n^{-1.271}$\\
\scriptsize (Altschuler and Parrilo \cite{altschuler:2024b},\\
\scriptsize Grimmer et al.~\cite{grimmer:2025,grimmer:2025a}, \\
\scriptsize Zhang et al.~\cite{zhang:2024})} & \makecell{$n^{-2}$ \\\scriptsize(Nemirovsky and Yudin \cite{nemirovsky:1983})} \\ \cline{2-4}
                                & \cmark & \makecell{\\[-8pt]$n^{-1.119}$\\\scriptsize(Zhang et al. \cite{zhang:2025})} & \makecell{\\[-8pt]$n^{-1.334}$\\\scriptsize(\Cref{thm:function_value_impossibility})} \\ \cline{1-4}
\multirow{2}{*}{$G_n$} & \xmark & \makecell{\rule{0pt}{1.3em}$n^{-1.271}$ \\ \scriptsize (Grimmer et al.~\cite{grimmer:2025,grimmer:2025a}, \\
\scriptsize Zhang et al.~\cite{zhang:2024})} & \makecell{$n^{-2}$\\\scriptsize(Nemirovsky and Yudin \cite{nemirovsky:1983})} \\ \cline{2-4}
                                & \cmark & \makecell{\\[-8pt]$n^{-1}$\\
                                \scriptsize (Taylor et al. \cite{taylor:2018})} & \makecell{\\[-8pt]$n^{-1}$\\\scriptsize(\Cref{thm:gradient_norm_impossibility})} \\
\bottomrule
\end{tabular}
\label{tab:sota_results}
\end{table}

\begin{table}[h!]
\centering
\caption{Summary of the best known convergence rates of first-order methods in smooth convex optimization.
Upper and lower bounds coincide in all settings except for the anytime convergence rate of $G_n$.
}
\def\arraystretch{2}
\begin{tabular}{@{} c|c|c|c }
\toprule
Setting & Anytime & Upper bound & Lower bound \\
\midrule
\multirow{2}{*}{$R_n$} & \xmark & \multirow{2}{*}{\makecell{$n^{-2}$ \\\scriptsize (Nesterov \cite{nesterov:1983})}}  & \multirow{4}{*}{\makecell{$n^{-2}$\\ \scriptsize(Nemirovsky and Yudin \cite{nemirovsky:1983})}} \\ \cline{2-2}
                                & \cmark & & \\ \cline{1-3}
\multirow{2}{*}{$G_n$} & \xmark & \makecell{\rule{0pt}{1.3em}$n^{-2}$ \\ \scriptsize (Kim and Fessler \cite{kim:2021}, \\ \scriptsize Kim et al.~\cite{kim:2023})} & \\ \cline{2-3}
                                & \cmark & \makecell{\\[-8pt]$n^{-1}$\\ \scriptsize (Taylor et al. \cite{taylor:2018})} & \\
\bottomrule
\end{tabular}
\label{tab:sota_results_first_order_methods}
\end{table}

\section{Helper Lemmas}\label{sec:helper_lemmas}

The quantity $Q_n(\schedule)$ can be written as the infinity norm $\norm{q_{\schedule}}_{[0,1]}$ of the polynomial $q_{\schedule}(\lambda) = \lambda\prod_{k=1}^n (1-\eta_k\lambda)^2$.
This observation connects quadratic optimization to polynomial approximation theory.
In this section, we collect related results used in our analysis.

\begin{definition}\label{def:chebyshev_polynomial}
The $n$-th Chebyshev polynomial of the first kind is defined as
\[
	T_n(x) = \begin{cases}
		\cos(n\arccos x) &\text{if } \abs{x}\leq 1, \\
		\cosh(n\cosh^{-1}x) &\text{if }x>1, \\
		(-1)^n\cosh(n\cosh^{-1}(-x)) &\text{if }x<-1. \\
	\end{cases}
\]
\end{definition}
One can show that $T_n$ defines an $n$-th degree real polynomial.
The Chebyshev polynomials satisfy a minimum-norm property:
$2^{1-n}T_n$ is the unique $n$-th degree monic polynomial that minimizes $\norm{ \cdot }_{[-1,1]}$ (see, e.g., Theorem~2.1.1 of Borwein and Erd\'{e}lyi~\cite{borwein:1995}).
The next lemma characterizes the minimum infinity norm over all polynomials $p$ satisfying $p(0)=1$.
The proof can be found in, for example, Chapter 2 of d'Aspremont et al.~\cite{daspremont:2021}.

\begin{lemma}\label{lem:minimum_norm_polynomial}
For any $L>\mu>0$ and any $n\in\N$, we have
\[
	\min_{ p\in \mathscr{P}_n: p(0)=1 } \norm{p}_{ [\mu,L] }
    = \frac{2}{\xi^n + \xi^{-n}},
\]
where $\xi = \frac{\sqrt{L}+\sqrt{\mu}}{\sqrt{L} - \sqrt{\mu}}$.
The minimum is achieved by $p(x)=\frac{T_n(\frac{2x-(L+\mu)}{L-\mu})}{T_n(-\frac{L+\mu}{L-\mu})}$.
\end{lemma}

The following lemma bounds the norm of the derivative $p'$ by the norm of $p$.
It is a shifted version of the standard Markov brothers' inequality (see, e.g., Theorem 5.1.8 of Borwein and Erd\'{e}lyi~\cite{borwein:1995}).

\begin{lemma}\label{lem:markov_brother}
For $p\in\mathscr{P}_n$, we have
\[
    \norm{ p' }_{[\mu, L]} \leq \frac{2n^2}{b-a} \norm{ p }_{[\mu, L]}.
\]
\end{lemma}

\begin{proof}
For any $p\in\mathscr{P}_n$ defined on $[\mu, L]$, let $q(x) = p\left( \frac{L-\mu}{2}x + \frac{L+\mu}{2} \right)$ for $x\in[-1,1]$.
We have $\deg q = \deg p = n$.
By Markov brothers' inequality,
\[
	\norm{ q' }_{[-1,1]} \leq n^2 \norm{ q }_{[-1,1]}.
\]
Since $\norm{ q' }_{[-1,1]} = \frac{L-\mu}{2}\norm{p'}_{[\mu,L]}$ and $\norm{ q }_{[-1,1]} = \norm{ p }_{[\mu,L]}$, we obtain
\[
	\norm{ p' }_{[\mu,L]} \leq \frac{2n^2}{L-\mu} \norm{ p }_{[\mu,L]}.
\]
This completes the proof.
\end{proof}

By combining \Cref{lem:minimum_norm_polynomial} and \Cref{lem:markov_brother}, we prove the following lemma.
Note that the constraint is $p'(0)=1$ instead of $p(0)=1$.

\begin{lemma}\label{lem:minimum_norm_polynomial_derivative}
For any $L>\mu>0$ and any $n\in\N$, we have
\[
	\min_{ p\in \mathscr{P}_n: p'(0)=1 } \norm{p}_{ [\mu,L] }
    \geq \frac{L-\mu}{2n^2}\cdot\xi^{-(n-1)}.
\]
where $\xi = \frac{\sqrt{L}+\sqrt{\mu}}{\sqrt{L} - \sqrt{\mu}}$.
\end{lemma}
\begin{proof}
By \Cref{lem:markov_brother}, we have
\[
	\min_{ p\in \mathscr{P}_n: p'(0)=1 } \norm{p}_{ [\mu,L] }
	\geq
	\frac{L-\mu}{2n^2} \cdot \min_{ p\in \mathscr{P}_n: p'(0)=1 } \norm{p'}_{ [\mu,L] }.
\]
By \Cref{lem:minimum_norm_polynomial}, we have
\begin{align*}
	\min_{ p\in \mathscr{P}_n: p'(0)=1 } \norm{p'}_{ [\mu,L] }
	&= \min_{ q\in \mathscr{P}_{n-1}: q(0)=1 } \norm{q}_{ [\mu,L] }  \\
	&= \frac{2}{\xi^{n-1} + \xi^{-(n-1)} } \\
	&\geq \frac{1}{\xi^{n-1}},
\end{align*}
where the inequality follows from $\xi^{-(n-1)} \leq \xi^{n-1}$.
The lemma follows by combining the above estimates.
\end{proof}

Lastly, we introduce the following lemma from measure theory, known as the layer-cake representation \cite{lieb:2001}.

\begin{lemma}\label{lem:layer_cake}
Let $(\R, \mathcal{A}, \mu)$ be a measure space.
For a non-negative measurable function $F:\R\to[0,\infty)$ and a measurable set $A\in\mathcal{A}$, we have
\[
	\int_{A} F(x)\, \du \mu(x) = \int_0^\infty \mu\left( \{ x\in A: F(x) > t \} \right)\,\du t.
\]
\end{lemma}

\end{document}